\documentclass[preprint,12pt]{elsarticle}
\journal{Discrete Applied Mathematics}
\usepackage[margin=1in]{geometry}
\newcommand{\lcom}{\underline{\ast}}
\usepackage{amsmath,amssymb,amsthm,mathtools}
\usepackage{bm}
\usepackage{tikz-cd}
\usepackage{forest}
\usepackage{multicol}

\usepackage{graphicx}
\usepackage{float}
\usepackage{booktabs}
\usepackage{enumitem}
\usepackage{xcolor}
\usepackage{microtype}
\usepackage{indentfirst}
\usepackage[colorlinks=true,linkcolor=blue,citecolor=blue,urlcolor=blue]{hyperref}
\usepackage[nameinlink,capitalise]{cleveref}

\theoremstyle{plain}
\newtheorem{theorem}{Theorem}
\newtheorem{lemma}[theorem]{Lemma}
\newtheorem{proposition}[theorem]{Proposition}

\theoremstyle{definition}
\newtheorem{definition}[theorem]{Definition}

\theoremstyle{remark}
\newtheorem{remark}[theorem]{Remark}

\date{}
\journal{Discrete Applied Mathematics}
\begin{document}
\begin{frontmatter}
\title{A Characterization of Trees with Trinomial Partial Petrial Polynomials}

\author[aff1]{Qingying Deng}
\ead{qingying@xtu.edu.cn}

\author[aff1]{Chaoyang Zhang\corref{cor1}}
\ead{202521511267@smail.xtu.edu.cn}

\cortext[cor1]{Corresponding author.}

\affiliation[aff1]{
  organization={School of Mathematics and Computational Science,
                Xiangtan University},
  city={Xiangtan~411105},
  state={Hunan},
  country={PR China}
}

\begin{abstract}
The partial Petrial polynomial of a bouquet can be computed from the
coranks over \(\mathrm{GF}(2)\) of matrices obtained by varying the
diagonal entries of the adjacency matrix of its intersection graph.
Motivated by this matrix formulation, we study the corresponding
polynomial for simple graphs and determine when it has exactly two or
three nonzero terms. A key tool is the interpolating property: the
exponents of the nonzero terms are consecutive. Using this property
together with local complementation minors of grafts, we extend the
known characterization of the binomial case from connected circle
graphs to all connected simple graphs, showing that such a graph has
a binomial partial Petrial polynomial if and only if it is a path.
Our main result characterizes the trinomial case for trees: a tree has
a trinomial partial Petrial polynomial if and only if it is a T-shape
tree or an H-shape tree. Here, a T-shape tree has maximum degree \(3\)
and exactly one vertex of degree \(3\), whereas an H-shape tree has
maximum degree \(3\) and exactly two vertices of degree \(3\), which
are adjacent. We also derive explicit formulas for both families in
terms of Jacobsthal numbers.
\end{abstract}

\begin{keyword}
partial Petrial polynomial, ribbon graph, local complementation minor, tree
\end{keyword}

\end{frontmatter}

\section{Introduction}
The Petrial operation was introduced by Wilson~\cite{Wilson1979Operators}.
For an embedded graph, it preserves the vertex and edge sets of the original
graph while replacing its facial walks with its Petrie polygons. In the
framework of ribbon graphs, the Petrial of a ribbon graph \(G\), denoted by
\(G^\times\), is obtained by inserting a half-twist into every edge band of
\(G\). More generally, for a subset \(A\subseteq E(G)\), the partial Petrial
\(G^{\times|A}\) is obtained by twisting precisely the edges in \(A\). Gross, Mansour, and Tucker~\cite{gross2021partial} introduced the partial Petrial polynomial as a generating function that enumerates the partial Petrials of a ribbon graph by Euler genus. They also established formulas and recursions for several families of ribbon graphs.

For a bouquet, the cyclic order of the loop ends around the unique vertex
determines its intersection graph, also called the unsigned interlace graph:
two loops correspond to adjacent vertices if their ends occur alternately
along the vertex boundary. Graphs arising as such intersection graphs are
precisely the circle graphs. Using this correspondence, Yan and
Li~\cite{YAN2025281} formulated the partial Petrial polynomial for circle
graphs and obtained explicit formulas for complete graphs and paths.

Feng, Yan, and Zheng~\cite{FENG2026411} characterized the binomial case for
connected circle graphs, proving that a connected circle graph has a partial
Petrial polynomial with exactly two nonzero terms if and only if it is a path.
Subsequently, Deng, Jin, and Yan~\cite{DengJinYan2026} extended the polynomial
to finite simple graphs by using an algebraic formulation in terms of coranks
over \(\mathrm{GF}(2)\), where the diagonal of the adjacency matrix is prescribed
by a chosen vertex subset. This formulation naturally leads to the question
whether the same binomial characterization remains valid for all connected
simple graphs. A structural ingredient needed for this question, and for the
trinomial classification below, is the interpolation property: the exponents
of its nonzero terms are consecutive. Using this property,
we first prove that the binomial characterization remains valid in the larger
setting: a connected simple graph has a partial Petrial polynomial with
exactly two nonzero terms if and only if it is a path.

Having settled the binomial case, we turn to the trinomial case. In this
paper, we study this problem for trees and give a complete answer: a tree has
a partial Petrial polynomial with exactly three nonzero terms precisely when
it is either a T-shape tree or an H-shape tree. Here, a T-shape tree is
obtained from one central vertex by attaching three pendant paths of positive
length, while an H-shape tree consists of two adjacent degree-\(3\) vertices,
each incident with two pendant paths of positive length. Precise definitions
are given in Section~\ref{sec:main-results}.

We also derive explicit formulas in
terms of the Jacobsthal numbers for the partial Petrial polynomials of \(T(a,b,c)\) and \(H(a,b,c,d)\), where \(a,b,c\) and \(a,b,c,d\) denote their respective branch lengths. These formulas show that all three coefficients are positive and
therefore establish the sufficiency part of the classification. The necessity part is proved by combining the adjacency-matrix formulation of
\(P_G^\times(z)\), local complementation minors of grafts, and a structural
reduction for trees. The key obstruction result shows that every tree which is
neither a path, nor a T-shape tree, nor an H-shape tree has a local
complementation minor containing at least three isolated vertices outside the
distinguished vertex subset. This forces the maximum corank of the adjacency
matrices associated with the corresponding grafts over \(\mathrm{GF}(2)\) to be
at least \(3\). Together with the interpolation property, this implies that
the polynomial has at least four nonzero terms, and hence rules out the
trinomial case.

\section{Preliminaries}
In this section, we recall the graph operations and algebraic
invariants used throughout the paper. Unless otherwise specified, all abstract graphs considered in this paper are finite, simple, and undirected.
We denote the path, complete graph, and star on \(n\) vertices by
\(P_n\), \(K_n\), and \(K_{1,n-1}\), respectively. All other graph-theoretic
notation follows standard conventions.

\begin{definition}[\cite{kotzig1968eulerian}]
Let \(G\) be a simple graph and \(v\in V(G)\). The \emph{local complementation}
at \(v\), denoted by \(G*v\), is the graph obtained from \(G\) by replacing
the induced subgraph on the neighborhood \(N_G(v)\) with its complement.
Equivalently, \(G*v\) is formed by toggling all adjacencies between vertices
in \(N_G(v)\). We further define
\[
G\underline{*}v := (G*v)\setminus\{v\}.
\]
\end{definition}

\begin{definition}
A \emph{graft} is a pair \((G,L_G)\), where \(G=(V,E)\) is a simple
graph with \(V=\{v_1,\ldots,v_n\}\) and \(L_G\subseteq V\). We write
\(\mathbf{A}_{(G,L_G)}=(a_{ij})_{n\times n}\) for its adjacency matrix over
\(\mathrm{GF}(2)\), where
\[
a_{ij}=
\begin{cases}
1, & \text{if } i\ne j \text{ and } v_iv_j\in E(G),\\
1, & \text{if } i=j \text{ and } v_i\in L_G,\\
0, & \text{otherwise}.
\end{cases}
\]
\end{definition}

The following definition adapts local complementation to grafts.

\begin{definition}[\cite{FENG2026411}]
Let \((G,L_G)\) be a graft. For a vertex \(v\in L_G\), the \emph{local
complementation} at \(v\) is the operation on the graft \((G,L_G)\) defined by
\[
(G,L_G)\mapsto (G*v,\,L_G\Delta N_G(v)).
\]
The \emph{local complementation deletion} at such a vertex \(v\) is the
operation
\[
(G,L_G)\underline{*}v
:=
\bigl(G\underline{*}v,\,(L_G\setminus\{v\})\Delta N_G(v)\bigr).
\]
A graft \((H,L_H)\) is a \emph{local complementation minor} of
\((G,L_G)\) if it can be obtained from \((G,L_G)\) by a sequence of local
complementation deletion operations.
\end{definition}

\begin{proposition}[\cite{FENG2026411}]
\label{prop:minor-corank}
Let \((G,L_G)\) and \((H,L_H)\) be grafts. If \((H,L_H)\) is a local
complementation minor of \((G,L_G)\), then
\[
\operatorname{corank}(\mathbf A_{(G,L_G)})
=
\operatorname{corank}(\mathbf A_{(H,L_H)}).
\]
\end{proposition}

\begin{theorem}[\upshape\cite{FENG2026411}]
\label{thm:not-path}
Let \(G\) be a simple graph that is not a path. Then there exists a subset
\(L_G\subseteq V(G)\) such that the graft \((G,L_G)\) contains a local
complementation minor \((G',L_{G'})\), where \(G'\) has at least two isolated
vertices, neither of which belongs to \(L_{G'}\).
\end{theorem}

We next recall the connection between bouquets and circle graphs. A \emph{bouquet} is a ribbon graph with exactly one vertex. For a bouquet \(B\), its \emph{intersection graph} \(I(B)\) is the simple graph whose vertices are the edges of \(B\), in which two vertices are adjacent precisely when the corresponding loop ends occur alternately around the unique vertex of \(B\). A graph is called a \emph{circle graph} if it is isomorphic to \(I(B)\) for some orientable bouquet \(B\).

The partial Petrial polynomial of a ribbon graph \(G\), introduced in~\cite{gross2021partial}, is defined by
\[
{}^\partial\varepsilon^\times_G(z)
:=
\sum_{A\subseteq E(G)}z^{\varepsilon(G^{\times|A})},
\]
where \(\varepsilon(G^{\times|A})\) is the Euler genus of
\(G^{\times|A}\).

\begin{definition}[\cite{YAN2025281}]
\label{def:circle-graph-partial-petrial-polynomial}
Let \(G\) be a circle graph. The partial Petrial polynomial of \(G\), denoted
by \(P_G^\times(z)\), is defined by
\[
P_G^\times(z)
:=
{}^\partial\varepsilon^\times_B(z),
\]
where \(B\) is an orientable bouquet whose intersection graph is isomorphic to
\(G\).
\end{definition}

\begin{lemma}[\cite{Mellor2003FewWeightSystems}]
\label{lem:faces-corank}
Let \(I(B)\) be the intersection graph of a bouquet \(B\), and let
\(S\subseteq V(I(B))\) correspond to the non-orientable loops of \(B\). Then
\[
f(B)
=
\operatorname{corank}(\mathbf A_{(I(B),S)})+1,
\]
where \(f(B)\) denotes the number of faces of \(B\).
\end{lemma}

The preceding lemma yields a matrix expression for the partial Petrial
polynomial of a circle graph. Indeed, if \(B\) is a bouquet with \(n\)
edges, then
\[
\varepsilon(B)=n+1-f(B).
\]
Now suppose that \(B\) is orientable and that \(A\subseteq E(B)\).
The partial Petrial \(B^{\times|A}\) has the same intersection graph
as \(B\), and its non-orientable loops correspond precisely to the
edges in \(A\). Hence, identifying \(E(B)\) with \(V(I(B))\), we obtain gives
\[
\varepsilon(B^{\times|A})
=
n-\operatorname{corank}\mathbf{A}_{(I(B),A)}.
\]
Consequently, for a circle graph \(G\) on \(n\) vertices,
Definition~\ref{def:circle-graph-partial-petrial-polynomial} can equivalently be
written as
\[
P^\times_G(z)
=
\sum_{L_G\subseteq V(G)}
z^{\,n-\operatorname{corank}\mathbf{A}_{(G,L_G)}}.
\]

The matrix expression above extends naturally to arbitrary simple
graphs.

\begin{definition}[\cite{DengJinYan2026}]
\label{def:partial-petrial-polynomial}
Let \(G=(V,E)\) be a simple graph on \(n\) vertices. The
\emph{partial Petrial polynomial} of \(G\) is defined by
\[
P_G^\times(z)
:=
\sum_{L_G\subseteq V}
z^{n-\operatorname{corank}(\mathbf A_{(G,L_G)})},
\]
where
\[
\operatorname{corank}(\mathbf A_{(G,L_G)})
=
n-\operatorname{rank}(\mathbf A_{(G,L_G)})
\]
is the nullity of the matrix over \(\mathrm{GF}(2)\).
\end{definition}

For circle graphs, the matrix definition above agrees with the original
definition via orientable bouquets. For non-circle graphs, it provides the
graph-theoretic extension used throughout the remainder of the paper.

The following two lemmas record the degree and interpolating property of \(P_G^\times(z)\).

\begin{lemma}[\cite{DengJinYan2026}]
\label{lem:degree-n}
For any simple graph \(G=(V,E)\) on \(n\) vertices, there exists a subset
\(L_G\subseteq V\) such that \(\mathbf A_{(G,L_G)}\) is nonsingular over
\(\mathrm{GF}(2)\). Consequently, the degree of the partial Petrial
polynomial \(P_G^\times(z)\) is exactly \(n\).
\end{lemma}

\begin{lemma}[\cite{DengJinYan2026}]
\label{lem:interpolating}
For any simple graph \(G\), the partial Petrial polynomial \(P_G^\times(z)\)
is an interpolating polynomial. That is, if \(P_G^\times(z)\) contains terms
\(z^k\) and \(z^m\) with \(k<m\), then it contains a term \(z^j\) with a
nonzero coefficient for every integer \(j\) such that \(k<j<m\).
\end{lemma}

It follows from
Lemmas~\ref{lem:degree-n} and~\ref{lem:interpolating} that the nonzero
terms of \(P_G^\times(z)\) occur consecutively from degree
\(n-\max_{L_G\subseteq V(G)}
\operatorname{corank}\mathbf{A}_{(G,L_G)}\) to degree \(n\).
Consequently, \(P_G^\times(z)\) is trinomial if and only if this
maximum corank is \(2\). Thus, the trinomial classification for trees
reduces to characterizing the trees that satisfy this condition.

\newpage
\section{Main Results}\label{sec:main-results}
We first extend the binomial characterization from connected circle graphs to all connected simple graphs and then turn to the trinomial characterization of trees.

\subsection{Binomial Characterization of Connected Simple Graphs}\label{subsec:binomial-characterization}
\begin{lemma}[\cite{YAN2025281}]\label{lem:path-formula}
Let \(P_n\) be a path with \(n\) vertices, where \(n\geq 1\). Then
\[
P_{P_n}^{\times}(z) =
\begin{cases}
\dfrac{2^n+1}{3}z^{n-1} + \dfrac{2^{n+1}-1}{3}z^n,
& \text{if } n \text{ is odd}, \\[1.2em]
\dfrac{2^n-1}{3}z^{n-1} + \dfrac{2^{n+1}+1}{3}z^n,
& \text{if } n \text{ is even}.
\end{cases}
\]
\end{lemma}

\begin{theorem}[\cite{FENG2026411}]\label{thm:path}
A connected circle graph \(G\) on \(n\) vertices has a binomial partial
Petrial polynomial if and only if \(G\) is a path.
\end{theorem}

The argument used in the proof of the preceding theorem cannot be
applied directly to arbitrary simple graphs. In that argument, the
circle-graph assumption is used precisely in passing from the
binomiality of \(P_G^\times(z)\) to the bound
\[
\operatorname{corank}\mathbf{A}_{(G,L_G)}\leq1
\qquad
\text{for every }L_G\subseteq V(G).
\]
Indeed, the proof in~\cite{FENG2026411} chooses an orientable bouquet
\(B\) whose intersection graph \(I(B)\) is isomorphic to \(G\).
After identifying \(E(B)\) with \(V(I(B))\), it applies
Lemma~\ref{lem:faces-corank} to express the Euler genus of each
partial Petrial \(B^{\times|D}\), \(D\subseteq E(B)\), in terms of
the corank of \(\mathbf{A}_{(I(B),D)}\). Consequently, the exponent
contributed by \(D\) is
\(n-\operatorname{corank}\mathbf{A}_{(I(B),D)}\). If \(G\) is not a
circle graph, no such bouquet exists, so this step is unavailable.

Theorem~\ref{thm:not-path}, however, already provides the required
structural obstruction for every simple graph that is not a path.
Moreover, Definition~\ref{def:partial-petrial-polynomial} expresses
\(P_G^\times(z)\) directly in terms of the coranks of
\(\mathbf{A}_{(G,L_G)}\) for every simple graph. Hence,
Lemmas~\ref{lem:degree-n} and~\ref{lem:interpolating} show that the
binomiality of \(P_G^\times(z)\) implies the displayed bound without
requiring a bouquet representation. Combining this bound with
Theorem~\ref{thm:not-path} yields the following extension.

\begin{theorem}\label{thm:simple-graph-path}
A connected simple graph \(G\) on \(n\) vertices has a binomial partial
Petrial polynomial if and only if \(G\) is a path.
\end{theorem}

\begin{proof}
\emph{Sufficiency.}
If \(G\) is a path, then
Lemma~\ref{lem:path-formula} shows that \(P_G^\times(z)\) has exactly
two nonzero terms. Hence \(P_G^\times(z)\) is binomial.

\emph{Necessity.}
Suppose that \(P_G^\times(z)\) is binomial. By
Lemmas~\ref{lem:degree-n} and~\ref{lem:interpolating}, its two
nonzero exponents are \(n-1\) and \(n\). It follows from
Definition~\ref{def:partial-petrial-polynomial} that
\begin{equation}\label{eq:corank-condition}
\operatorname{corank}\bigl(\mathbf{A}_{(G,L_G)}\bigr)\leq1
\qquad
\text{for every }L_G\subseteq V(G),
\end{equation}
since any \(L_G\) for which the corank is at least \(2\) would
contribute a term of degree at most \(n-2\).

Suppose, to the contrary, that \(G\) is not a path. By
Theorem~\ref{thm:not-path}, there exists \(L_G\subseteq V(G)\) such
that \((G,L_G)\) has a local complementation minor
\((G',L_{G'})\), where \(G'\) has at least two isolated vertices
outside \(L_{G'}\). Let \(u\) and \(w\) be two such vertices. The rows
of \(\mathbf{A}_{(G',L_{G'})}\) indexed by \(u\) and \(w\) are zero,
and hence
\[
\operatorname{corank}\bigl(\mathbf{A}_{(G',L_{G'})}\bigr)\geq2.
\]
By Proposition~\ref{prop:minor-corank},
\[
\operatorname{corank}\bigl(\mathbf{A}_{(G,L_G)}\bigr)
=
\operatorname{corank}\bigl(\mathbf{A}_{(G',L_{G'})}\bigr)
\geq2,
\]
contradicting~\eqref{eq:corank-condition}. Therefore \(G\) is a path.
\end{proof}

\subsection{Trinomial Characterization of Trees}
\label{subsec:trinomial-characterization}

We now turn to trees with trinomial partial Petrial polynomials. By
Definition~\ref{def:partial-petrial-polynomial} and
Lemmas~\ref{lem:degree-n} and~\ref{lem:interpolating}, a tree \(T\)
has a trinomial partial Petrial polynomial if and only if
\[
\max_{L_T\subseteq V(T)}
\operatorname{corank}\mathbf{A}_{(T,L_T)}=2.
\]
We therefore seek to characterize the trees satisfying this equality.

For positive integers \(a,b,c\), a \emph{T-shape tree}
\(T(a,b,c)\)~\cite{wang2014signless} is obtained by attaching three pendant paths
of lengths \(a,b,c\) to a central vertex. Its degree sequence is
\((3,2,\ldots,2,1,1,1)\).

Similarly, for positive integers \(a,b,c,d\), an \emph{H-shape tree}
\(H(a,b,c,d)\)~\cite{chen2023divisibility} consists of two adjacent vertices of degree
\(3\), one incident with two pendant paths of lengths \(a\) and \(b\),
and the other with two pendant paths of lengths \(c\) and \(d\). Its
degree sequence is
\((3,3,2,\ldots,2,1,1,1,1)\).

The numbers of vertices of degree \(2\) in \(T(a,b,c)\) and
\(H(a,b,c,d)\) are \(a+b+c-3\) and \(a+b+c+d-4\), respectively.
Thus, in either degree sequence, the string \(2,\ldots,2\) may be
empty.

We can now state the main classification theorem.

\begin{theorem}\label{thm:structure}
    A tree \(T\) on \(n\) vertices has a trinomial partial Petrial polynomial if and only if \(T\) is a T-shape tree or an H-shape tree.
\end{theorem}

By Lemma~\ref{lem:path-formula}, paths have binomial partial Petrial
polynomials. Let \(\mathcal E\) denote the family consisting of paths,
T-shape trees, and H-shape trees. To prove the necessity part of
Theorem~\ref{thm:structure}, it remains to show that every tree
\(T\notin\mathcal E\) satisfies
\[
\max_{L_T\subseteq V(T)}
\operatorname{corank}\mathbf A_{(T,L_T)}\geq3.
\]
The following theorem provides the required structural obstruction.

\begin{theorem}\label{th:main}
    Let $T$ be a tree on \(n\) vertices with \(T\notin \mathcal{E}\). Then there exists $L_T \subseteq V(T)$ such that the graft $(T,L_T)$ contains a local complementation minor $(T',L_{T'})$ where $T'$ has at least three isolated vertices, none of which belongs to $L_{T'}$.
\end{theorem}

\begin{lemma}\label{lem:leaf-deletion-lifting}
Let \(T'\) be obtained from a tree \(T\) by a sequence of leaf
deletions. If \(T'\) admits a local complementation minor as in the preceding theorem, then so does \(T\).
\end{lemma}

\begin{proof}
It suffices to consider one local complementation deletion at a leaf. Let \(x\) be a leaf of \(T\), let
\(y\) be its unique neighbour, and put \(T'=T\lcom x\). Since \(x\) is a leaf, local complementation deletion at \(x\) toggles no edges, and therefore
\[
    T\lcom x=T-x.
\]

Let \(L_{T'}\subseteq V(T')\). Set
\[
    L_T=\{x\}\cup (L_{T'}\triangle\{y\}).
\]
Then \(x\in L_T\), so local complementation deletion at \(x\) is allowed.
Moreover, the resulting subset is
\[
    (L_T\setminus\{x\})\triangle N_T(x)
    =
    (L_{T'}\triangle\{y\})\triangle\{y\}
    =
    L_{T'}.
\]
Hence
\[
    (T,L_T)\lcom x=(T',L_{T'}).
\]
Thus every subset \(L_{T'}\subseteq V(T')\) can be lifted to a subset
\(L_T\subseteq V(T)\). Iterating this argument along the sequence of leaf
deletions proves the claim.
\end{proof}

\begin{lemma}\label{lem:exceptional case}
    Let \(T\) be a tree on \(n\) vertices with  \(T \notin \mathcal{E}\). Suppose that at least one of the following conditions holds:
    \begin{enumerate}
        \item \(T \cong K_{1,n-1}\) for \(n \ge 5\);
        \item \(T\) contains a vertex \(v\) of degree \(2\) such that \(T \lcom v \in \mathcal{E}\);
        \item \(T\) contains no vertex of degree \(2\), and for every leaf \(v\) of \(T\), one has 
        \(T\lcom v=T-v\in \mathcal{E}.\)
    \end{enumerate}
    Then there exists \(L_T\subseteq V(T)\) such that the graft \((T,L_T)\) contains a local complementation minor \((T',L_{T'})\), where \(T'\) has at least three isolated vertices, none of which belongs to \(L_{T'}\).
\end{lemma}

\begin{proof}
We use Lemma~\ref{lem:leaf-deletion-lifting} above whenever pendant branches are shortened by local complementation deletions at leaves.
Let \(V(T)=\{v_1,\ldots,v_n\}.\)

\medskip
\textbf{Case 1.} Suppose that \(T \cong K_{1,n-1}\), where \(n \ge 5\). Let \(v_1\) be the central vertex of \(T\), and let
\[N_T(v_1)=\{v_2,\ldots,v_n\}.\]
Set \(L_T=\{v_1\}\). Then
\[(T,L_T) \lcom v_1 \lcom v_2=(T',\emptyset) \]
where \(T'=T\lcom v_1 \lcom v_2\) is \(\overline{K_{n-2}}\) and contains at least three isolated vertices, none of which belongs to the final subset.

\medskip
\textbf{Case 2.} Suppose that \(T\) contains a vertex \(v_i\) of degree \(2\) such that
\(T\lcom v_i\in\mathcal E\). Let
\[
    N_T(v_i)=\{v_j,v_k\}.
\]
Since \(v_i\) has degree \(2\), the underlying graph \(T\lcom v_i\) is
obtained from \(T\) by suppressing \(v_i\); equivalently, \(T\) is obtained
from \(T\lcom v_i\) by subdividing the edge \(v_jv_k\).

If \(T\lcom v_i\cong P_{n-1}\), then \(T\cong P_n\), contradicting
\(T\notin\mathcal E\).

If \(T\lcom v_i\) is a T-shape tree, then \(T\) is also a T-shape tree,
since subdividing an edge of a T-shape tree only increases the length of one
of its three branches. This contradicts \(T\notin\mathcal E\).

If \(T\lcom v_i\) is an H-shape tree, then \(T\) is obtained from an
H-shape tree by subdividing one of its edges. If the subdivided edge lies on
one of the four pendant branches, then \(T\) is again an H-shape tree,
contrary to \(T\notin\mathcal E\). Hence the subdivided edge must be the
central edge between the two degree-\(3\) vertices.

By Lemma~\ref{lem:leaf-deletion-lifting}, we shorten the four pendant branches to length \(1\) by local complementation deletions at leaves. After relabelling the reduced core as in
Figure~\ref{fig:T_1}, we obtain the tree \(T_1\). 
For this reduced core, set
\[
    L_{T_1}=\{v_2,v_5\}.
\]
Then
\[
    (T_1,L_{T_1})\lcom v_2\lcom v_5\lcom v_1\lcom v_4
    =
    (T_1',\emptyset),
\]
where \(T_1'\) has three isolated vertices, none of which belongs to the final subset. By Lemma~\ref{lem:leaf-deletion-lifting}, this local complementation minor lifts to one for the original tree \(T\).

\medskip
\textbf{Case 3.} Suppose that \(T\) contains no vertex of degree \(2\), and that for
every leaf \(v_i\) of \(T\), one has
\[
    T\lcom v_i=T-v_i\in\mathcal E .
\]
Let \(v_i\) be a leaf of \(T\), and let \(v_j\) be its unique neighbour.
Since \(v_i\) is a leaf, we have
\[
    T\lcom v_i=T-v_i .
\]
We distinguish the three possibilities for \(T-v_i\).

If \(T-v_i\cong P_{n-1}\), then \(T\) is obtained from a path by attaching
one pendant vertex. Hence \(T\) is either a path or a T-shape tree, contrary
to \(T\notin\mathcal E\).

If \(T-v_i\) is a T-shape tree, then \(T\) is obtained from a T-shape tree
by attaching one pendant vertex. If the attachment is on one of its
three branches, then \(T\) is an H-shape tree, contrary to
\(T\notin\mathcal E\). If the attachment is at the central vertex, then
the central vertex has degree \(4\). By reducing the four branches incident
with this vertex to length \(1\), we obtain a star \(K_{1,4}\) as a
local complementation minor. Hence the conclusion follows from the
star case proved above.

It remains to consider the case where \(T-v_i\) is an H-shape tree. Then
\(T\) is obtained from an H-shape tree by attaching one pendant vertex.

If the attachment is at one of the two degree-\(3\) vertices of
\(T-v_i\), then that vertex becomes a vertex of degree \(4\) in \(T\).
Deleting the extra pendant leaves on the other side, if necessary, and then
reducing the four branches incident with this degree \(4\) vertex to length
\(1\), we again obtain a star \(K_{1,4}\) as a local complementation 
minor. Thus the conclusion follows from the star case.

We are left with the case where the attachment is on one of the four
pendant branches of the H-shape tree. Since \(T\) contains no vertex of
degree \(2\), by Lemma~\ref{lem:leaf-deletion-lifting}, it is enough to consider the reduced core shown in Figure~\ref{fig:T_2}.

For the reduced core in Figure~\ref{fig:T_2}, set \(L_{T_2}=\{v_2,v_4,v_5,v_6\}.\) Then
\[
    (T_2,L_{T_2})\lcom v_2\lcom v_6\lcom v_1\lcom v_7\lcom v_4
    =
    (T_2',\emptyset),
\]
where \(T_2'\) has three isolated vertices, none of which belongs to the final subset.

By Lemma~\ref{lem:leaf-deletion-lifting}, the local complementation minor obtained from the reduced core lifts to one for the original tree \(T\). This proves the claim in the final case.

\begin{figure}[H]
\centering
\begin{minipage}{0.45\textwidth}
    \centering
    \begin{tikzpicture}[every node/.style={circle, fill, inner sep=2pt}]
        \node (v4) at (0,0) [label=below:$v_4$] {};
        
        \node (v2) at (-1.5,0) [label=below right:$v_2$] {};
        \node (v5) at (1.5,0) [label=below left:$v_5$] {};
        
        \node (v1) at (-2.3,1) [label=above:$v_1$] {};
        \node (v3) at (-2.3,-1) [label=below:$v_3$] {};
        
        \node (v6) at (2.3,1) [label=above:$v_6$] {};
        \node (v7) at (2.3,-1) [label=below:$v_7$] {};
        
        \draw (v4) -- (v2);
        \draw (v4) -- (v5);
        \draw (v2) -- (v1);
        \draw (v2) -- (v3);
        \draw (v5) -- (v6);
        \draw (v5) -- (v7);
    \end{tikzpicture}
    \caption{$T_1$}
    \label{fig:T_1}
\end{minipage}
\hfill
\begin{minipage}{0.45\textwidth}
    \centering
    \begin{forest}
      for tree={
        circle, fill, inner sep=2pt,
        l=1cm, l sep=0.8cm,
        parent anchor=center,
        child anchor=center,
        s sep=0.6cm
      }
      [ , label=right:$v_1$
        [ , label=right:$v_2$
          [ , label=right:$v_3$]
          [ , label=right:$v_4$
            [ , label=right:$v_5$]
            [ , label=right:$v_6$
              [ , label=right:$v_7$]
              [ , label=right:$v_8$]
            ]
          ]
        ]
      ]
    \end{forest}
    \caption{$T_2$}
    \label{fig:T_2}
\end{minipage}
\end{figure}
\end{proof}

\begin{proof}[Proof of Theorem \ref{th:main}]

We argue by strong induction on the number of vertices \(n=|V(T)|\).
If \(n\le 4\), then every tree on \(n\) vertices belongs to \(\mathcal E\).
Hence there is no tree \(T\notin\mathcal E\) to consider.

Assume now that \(n\ge 5\), and suppose that the theorem holds for every tree with fewer than \(n\) vertices that does not belong to \(\mathcal E\). Let \(T\) be a tree on \(n\) vertices with \(T\notin\mathcal E\).

If \(T\cong K_{1,n-1}\), then the conclusion follows from
Lemma~\ref{lem:exceptional case}. Hence we may assume that
\(T\not\cong K_{1,n-1}\).

Suppose first that \(T\) contains a vertex of degree \(2\). Let \(v_1\) be
such a vertex, and let
\[
    N_T(v_1)=\{v_2,v_3\}.
\]
Since \(v_1\) has degree \(2\), the local complementation deletion at \(v_1\)
suppresses \(v_1\); equivalently, \(T\lcom v_1\) is the tree obtained from
\(T\) by replacing the path \(v_2v_1v_3\) with the edge \(v_2v_3\). Thus
\(T\lcom v_1\) has \(n-1\) vertices.

If \(T\lcom v_1\notin\mathcal E\), then the induction hypothesis applies to
\(T\lcom v_1\). Thus there exists \(L_1\subseteq V(T\lcom v_1)\) such that
\((T\lcom v_1,L_1)\) contains a local complementation minor
\((T',L_{T'})\), where \(T'\) has at least three isolated vertices, none of which belongs to \(L_{T'}\). Hence the desired conclusion holds for \(T\).

If \(T\lcom v_1\in\mathcal E\), then the conclusion follows from
Lemma~\ref{lem:exceptional case}.

We may therefore assume that \(T\) contains no vertex of degree \(2\).

Since local complementation deletion at a leaf agrees with ordinary deletion, suppose that there exists a leaf \(v\) such that
\[
    T\lcom v=T-v \notin \mathcal{E}.
\]
Then the induction hypothesis applies to \(T-v\). Hence there exists
\(L_1\subseteq V(T-v)\) such that \((T-v,L_1)\) contains a local
complementation minor \((T',L_{T'})\), where \(T'\) has at least three
isolated vertices, none of which belongs to \(L_{T'}\). By Lemma~\ref{lem:leaf-deletion-lifting}, the conclusion also holds for \(T\).

It remains to consider the case where \(T\lcom v=T-v\in\mathcal{E}\) for every leaf \(v\) of \(T\). Thus \(T\) satisfies the third condition in Lemma~\ref{lem:exceptional case}, and the conclusion follows.

This completes the induction, and hence the theorem follows.

\end{proof}

\begin{lemma}\label{lem:jacobsthal_id}
Let \(\{J_k\}_{k\geq0}\) be the Jacobsthal sequence defined by
\[
J_0=0,\qquad J_1=1,\qquad
J_k=J_{k-1}+2J_{k-2}\quad(k\geq2).
\]
Equivalently,
\[
J_k=\frac{2^k-(-1)^k}{3}\quad(k\geq0).
\]
For any integers \(x,y\geq0\), the following identities hold:
\begin{enumerate}
    \item[(i)] \(J_{x+y+1}=J_{x+1}J_{y+1}+2J_xJ_y;\)

    \item[(ii)] \(2(J_xJ_{y+1}+J_{x+1}J_y)=J_{x+y+2}-J_{x+1}J_{y+1}.\)
   
\end{enumerate}
\end{lemma}

\begin{proof}
Using the closed form of \(J_k\), we obtain
\begin{align*}
J_{x+1}J_{y+1}+2J_xJ_y
&=
\frac{1}{9}\Bigl[
(2^{x+1}-(-1)^{x+1})
(2^{y+1}-(-1)^{y+1})  \\
&\hspace{3.2cm}
+2(2^x-(-1)^x)(2^y-(-1)^y)
\Bigr] \\
&=
\frac{2^{x+y+1}-(-1)^{x+y+1}}{3}
=
J_{x+y+1}.
\end{align*}

For (ii), applying (i) to \((x,y+1)\) and using
\(J_{y+2}=J_{y+1}+2J_y\), we obtain
\begin{align*}
J_{x+y+2}
&=
J_{x+1}J_{y+2}+2J_xJ_{y+1} \\
&=
J_{x+1}(J_{y+1}+2J_y)+2J_xJ_{y+1} \\
&=
J_{x+1}J_{y+1}
+2(J_xJ_{y+1}+J_{x+1}J_y).
\end{align*}
Rearranging proves (ii).
\end{proof}

\begin{lemma}[\cite{DengJinYan2026}]\label{lem:leaf_reduction}
Let \(G\) be a simple graph containing a leaf \(x\), and let \(y\) be
the unique neighbour of \(x\). Then
\begin{equation}\label{eq:leaf_recurrence}
P^\times_G(z)
=
zP^\times_{G-x}(z)
+
2z^2P^\times_{G-x-y}(z).
\end{equation}
\end{lemma}

\begin{lemma}\label{lem:leaf-jacobsthal-recurrence}
Let \(G_0\) be a simple graph and let \(v\in V(G_0)\). Put
\(G_{-1}=G_0-v\). For \(r\geq1\), let \(G_r\) be obtained from
\(G_0\) by identifying one endpoint of a path of length \(r\) with
\(v\). Then,
for every \(r\geq0\),
\[
P^\times_{G_r}(z)
=
z^r\left(
J_{r+1}P^\times_{G_0}(z)
+
2J_rzP^\times_{G_{-1}}(z)
\right),
\]
where \(J_r\) is the \(r\)th Jacobsthal number.
\end{lemma}

\begin{proof}
Write \(X_r=P^\times_{G_r}(z)\). Applying
Lemma~\ref{lem:leaf_reduction} successively to the terminal leaves of
the attached path gives
\[
X_r=zX_{r-1}+2z^2X_{r-2}
\qquad (r\geq1),
\]
with initial terms \(X_{-1}=P^\times_{G_{-1}}(z)\) and
\(X_0=P^\times_{G_0}(z)\).

The characteristic equation is
\[
\lambda^2-z\lambda-2z^2=0,
\]
with roots \(2z\) and \(-z\). Determining the constants from
\(X_{-1}\) and \(X_0\) gives
\[
X_r
=
J_{r+1}z^rX_0
+
2J_rz^{r+1}X_{-1}.
\]
Substituting back \(X_r=P^\times_{G_r}(z)\) proves the result.
\end{proof}

\begin{remark}[\cite{DengJinYan2026}]
For a disconnected simple graph \(G=G_1\cup G_2\), the vertex subsets
in \(G_1\) and \(G_2\) are chosen independently. Hence
\[
P^\times_{G_1\cup G_2}(z)
=
P^\times_{G_1}(z)P^\times_{G_2}(z).
\]
\end{remark}

\begin{theorem}\label{thm:T-shape_explicit}
Let \(T(a,b,c)\) be the T-shape tree with \(a,b,c\ge 1\), consisting of a
central vertex together with three pendant paths of lengths \(a\), \(b\), and
\(c\). Put \(n=a+b+c+1\), where \(J_k\) is as in Lemma~\ref{lem:jacobsthal_id}.

Then
\[
    P^\times_{T(a,b,c)}(z)=Az^n+Bz^{n-1}+Cz^{n-2},
\]
where
\begin{align*}
    A&=J_{a+b+2}J_{c+1}+2J_{a+1}J_{b+1}J_c,\\
    B&=J_{a+b+1}J_{c+1}+2\bigl(J_aJ_{b+1}+J_{a+1}J_b\bigr)J_c,\\
    C&=2J_aJ_bJ_c.
\end{align*}
\end{theorem}

\begin{proof}
The tree \(T(a,b,c)\) is obtained from the path \(P_{a+b+1}\) by attaching a pendant path of length \(c\) to the vertex at distance \(a\) from one end of the path. We vary the length of this attached path. For \(r\ge 1\), let \(T(a,b,r)\) denote the corresponding T-shape tree with attached path of length \(r\). Set
\[
    T(a,b,0):=P_{a+b+1},
    \qquad
    T(a,b,-1):=P_a\cup P_b.
\]
The graph \(T(a,b,-1)\) is obtained from \(T(a,b,1)\) by deleting the leaf on the third branch together with its unique neighbour, namely the central vertex.

For every \(r\ge 1\), let \(x_r\) be the leaf at the end of the third branch
of \(T(a,b,r)\), and let \(y_r\) be its unique neighbour. Then
\[
    T(a,b,r)-x_r=T(a,b,r-1),
    \qquad
    T(a,b,r)-\{x_r,y_r\}=T(a,b,r-2).
\]
Hence Lemma~\ref{lem:leaf-jacobsthal-recurrence} gives
\[
P^\times_{T(a,b,c)}(z)=z^c\left(
J_{c+1}P^\times_{T(a,b,0)}(z)+2J_czP^\times_{T(a,b,-1)}(z)
\right).
\]

We now compute \(P^\times_{T(a,b,0)}(z)\) and \(P^\times_{T(a,b,-1)}(z)\). Since \(T(a,b,0)\cong P_{a+b+1},\) Lemma~\ref{lem:path-formula} gives
\[
P^\times_{T(a,b,0)}(z)
=J_{a+b+2}z^{a+b+1}+J_{a+b+1}z^{a+b}.
\]
On the other hand, by multiplicativity over disjoint unions,
\[
\begin{aligned}
P^\times_{T(a,b,-1)}(z)&=P^\times_{P_a\cup P_b}(z)\\
&=P^\times_{P_a}(z)P^\times_{P_b}(z)\\
&=\bigl(J_{a+1}z^a+J_az^{a-1}\bigr)\bigl(J_{b+1}z^b+J_bz^{b-1}\bigr)\\
&=J_{a+1}J_{b+1}z^{a+b}+\bigl(J_aJ_{b+1}+J_{a+1}J_b\bigr)z^{a+b-1}
+J_aJ_bz^{a+b-2}.
\end{aligned}
\]

Substituting these expressions into the Jacobsthal formula yields
\[
\begin{aligned}
P^\times_{T(a,b,c)}(z)
&=
z^cJ_{c+1}
\bigl(J_{a+b+2}z^{a+b+1}+J_{a+b+1}z^{a+b}
\bigr)\\
&\quad
+2J_cz^{c+1}
\bigl(J_{a+1}J_{b+1}z^{a+b}
+\bigl(J_aJ_{b+1}+J_{a+1}J_b\bigr)z^{a+b-1}+J_aJ_bz^{a+b-2}
\bigr)\\
&=
\bigl(J_{a+b+2}J_{c+1}+2J_{a+1}J_{b+1}J_c
\bigr)z^{a+b+c+1}\\
&\quad
+\bigl(J_{a+b+1}J_{c+1}+2\bigl(J_aJ_{b+1}+J_{a+1}J_b\bigr)J_c
\bigr)z^{a+b+c}\\
&\quad
+2J_aJ_bJ_cz^{a+b+c-1}.
\end{aligned}
\]
Since \(n=a+b+c+1\), this is precisely
\[
    P^\times_{T(a,b,c)}(z)
    =
    Az^n+Bz^{n-1}+Cz^{n-2},
\]
with
\[
A=J_{a+b+2}J_{c+1}+2J_{a+1}J_{b+1}J_c,
\]
\[
B=J_{a+b+1}J_{c+1}+2\bigl(J_aJ_{b+1}+J_{a+1}J_b\bigr)J_c,
\]
and
\[
C=2J_aJ_bJ_c.
\]
The proof is complete.
\end{proof}

\begin{theorem}\label{thm:H-shape_trinomial}
Let \(H(a,b,c,d)\) be the H-shape tree with \(a,b,c,d\ge 1\), having two
adjacent degree-\(3\) vertices \(u\) and \(v\), where the two pendant paths
incident with \(u\) have lengths \(a\) and \(b\), and the two pendant paths
incident with \(v\) have lengths \(c\) and \(d\). Put \(n=a+b+c+d+2\), where \(J_k\) is as in Lemma~\ref{lem:jacobsthal_id}.

Then
\[
    P^\times_{H(a,b,c,d)}(z)=Az^n+Bz^{n-1}+Cz^{n-2},
\]
where
\begin{align*}
    A&=J_{a+b+2}J_{c+d+2}+2J_{a+1}J_{b+1}J_{c+1}J_{d+1},\\
    B&=J_{a+b+1}J_{c+d+2}+J_{a+b+2}J_{c+d+1}-J_{a+1}J_{b+1}J_{c+1}J_{d+1},\\
    C&=J_{a+b+1}J_{c+d+1}-J_{a+1}J_{b+1}J_{c+1}J_{d+1}.
\end{align*}
\end{theorem}

\begin{proof}
Fix \(a,b,c\), and vary the length of the fourth branch. For \(r\ge 1\),
let \(H(a,b,c,r)\) denote the H-shape tree whose fourth branch has length
\(r\). For this proof, also set
\[
    H(a,b,c,0):=T(a,b,c+1),
    \qquad
    H(a,b,c,-1):=P_c\cup P_{a+b+1}.
\]

For every \(r\ge 1\), let \(x_r\) be the leaf at the end of the fourth branch
of \(H(a,b,c,r)\), and let \(y_r\) be its unique neighbour. Then
\[
    H(a,b,c,r)-x_r=H(a,b,c,r-1),
\]
and
\[
    H(a,b,c,r)-\{x_r,y_r\}=H(a,b,c,r-2).
\]
Applying Lemma~\ref{lem:leaf-jacobsthal-recurrence} and then taking \(r=d\),
we obtain
\[
P^\times_{H(a,b,c,d)}(z)
=
z^d\left(
J_{d+1}P^\times_{H(a,b,c,0)}(z)
+
2J_dzP^\times_{H(a,b,c,-1)}(z)
\right).
\]

It remains to compute \(P^\times_{H(a,b,c,0)}(z)\) and \(P^\times_{H(a,b,c,-1)}(z)\). Since \(H(a,b,c,0)=T(a,b,c+1),\) Theorem~\ref{thm:T-shape_explicit} gives
\[
\begin{aligned}
P^\times_{H(a,b,c,0)}(z)
&=\bigl(J_{a+b+2}J_{c+2}
+2J_{a+1}J_{b+1}J_{c+1}\bigr)z^{a+b+c+2}\\
&\quad
+\bigl(J_{a+b+1}J_{c+2}
+2\bigl(J_aJ_{b+1}+J_{a+1}J_b\bigr)J_{c+1}\bigr)z^{a+b+c+1}\\
&\quad
+2J_aJ_bJ_{c+1}z^{a+b+c}.
\end{aligned}
\]
On the other hand, by multiplicativity over disjoint unions,
\[
\begin{aligned}
P^\times_{H(a,b,c,-1)}(z)
&=P^\times_{P_c\cup P_{a+b+1}}(z)\\
&=P^\times_{P_c}(z)P^\times_{P_{a+b+1}}(z)\\
&=\bigl(J_{c+1}z^c+J_cz^{c-1}\bigr)
  \bigl(J_{a+b+2}z^{a+b+1}+J_{a+b+1}z^{a+b}\bigr)\\
&=J_{c+1}J_{a+b+2}z^{a+b+c+1}\\
&\quad
+\bigl(J_cJ_{a+b+2}+J_{c+1}J_{a+b+1}\bigr)z^{a+b+c}\\
&\quad
+J_cJ_{a+b+1}z^{a+b+c-1}.
\end{aligned}
\]

Substituting these expressions into the preceding formula and collecting
terms gives
\[
P^\times_{H(a,b,c,d)}(z)=Az^n+Bz^{n-1}+Cz^{n-2},
\]
where \(n=a+b+c+d+2\) and
\[
\begin{aligned}
A
&=J_{d+1}\bigl(J_{a+b+2}J_{c+2}
+2J_{a+1}J_{b+1}J_{c+1}\bigr)
+2J_dJ_{c+1}J_{a+b+2},\\
B
&=J_{d+1}\bigl(J_{a+b+1}J_{c+2}
+2\bigl(J_aJ_{b+1}+J_{a+1}J_b\bigr)J_{c+1}\bigr)\\
&\quad
+2J_d\bigl(J_cJ_{a+b+2}+J_{c+1}J_{a+b+1}\bigr),\\
C
&=2J_{d+1}J_aJ_bJ_{c+1}
+2J_dJ_cJ_{a+b+1}.
\end{aligned}
\]
Using Lemma~\ref{lem:jacobsthal_id} in the forms
\[
    J_{c+d+2}=J_{c+2}J_{d+1}+2J_{c+1}J_d,
    \qquad
    J_{c+d+1}=J_{c+1}J_{d+1}+2J_cJ_d,
\]
\[
    2\bigl(J_aJ_{b+1}+J_{a+1}J_b\bigr)
    =
    J_{a+b+2}-J_{a+1}J_{b+1},
    \qquad
    2J_aJ_b=J_{a+b+1}-J_{a+1}J_{b+1},
\]
these coefficients simplify to
\[
\begin{aligned}
A&=J_{a+b+2}J_{c+d+2}
+2J_{a+1}J_{b+1}J_{c+1}J_{d+1},\\
B&=J_{a+b+1}J_{c+d+2}
+J_{a+b+2}J_{c+d+1}
-J_{a+1}J_{b+1}J_{c+1}J_{d+1},\\
C&=J_{a+b+1}J_{c+d+1}
-J_{a+1}J_{b+1}J_{c+1}J_{d+1}.
\end{aligned}
\]
This gives the asserted formulas for \(A\), \(B\), and \(C\).
\end{proof}

\textbf{Proof of Theorem~\ref{thm:structure}}

\textit{Sufficiency.} Suppose first that \(T\) is a T-shape tree or an H-shape tree. By
Theorems~\ref{thm:T-shape_explicit} and~\ref{thm:H-shape_trinomial},
respectively, \(P_T^\times(z)\) has the form
\[
P_T^\times(z)=Az^n+Bz^{n-1}+Cz^{n-2}.
\]
Since \(J_k>0\) for every \(k\geq1\), the formulas in Theorem~\ref{thm:T-shape_explicit} and the positive-sum expressions obtained in the proof of
Theorem~\ref{thm:H-shape_trinomial} show that \(A,B,C>0\). Hence \(P_T^\times(z)\) is trinomial.

\textit{Necessity.} Conversely, suppose that \(P_T^\times(z)\) is trinomial. Since
\(P_T^\times(z)\) has degree \(n\) and is interpolating, its nonzero terms
must be exactly
\[
z^{n-2},\quad z^{n-1},\quad z^n.
\]
Therefore, by Definition~\ref{def:partial-petrial-polynomial},
\[
\operatorname{corank}(\mathbf A_{(T,L)})\leq 2
\]
for every \(L\subseteq V(T)\).

If \(T\) is a path, then \(P_T^\times(z)\) is binomial by
Lemma~\ref{lem:path-formula}, contradicting the assumption that it is
trinomial. Hence \(T\) is not a path.

Assume now that \(T\) is neither a T-shape tree nor an H-shape tree. By
Theorem~\ref{th:main}, there exists \(L_T\subseteq V(T)\) such
that the graft \((T,L_T)\) contains a local complementation minor
\((T',L_{T'})\), where \(T'\) has at least three isolated vertices, none of
which belongs to \(L_{T'}\). Consequently, \(\mathbf A_{(T',L_{T'})}\) has at least three zero rows, and hence
\[
\operatorname{corank}\mathbf A_{(T',L_{T'})}\geq 3.
\]
By Proposition~\ref{prop:minor-corank},
\[
\operatorname{corank}\mathbf A_{(T,L_T)}
=
\operatorname{corank}\mathbf A_{(T',L_{T'})}
\geq 3,
\]
contradicting \(\operatorname{corank}\mathbf A_{(T,L)}\leq 2\) for all
\(L\subseteq V(T)\). Therefore \(T\) must be a T-shape tree or an H-shape
tree.

\newpage
\bibliographystyle{plain} 
\bibliography{references}

\end{document}